\documentclass[11pt,centertags]{amsart}

\usepackage[utf8]{inputenc}

\usepackage{amssymb}
\usepackage{amsthm}
\usepackage{stmaryrd}

\usepackage{tikz-cd}
\usetikzlibrary{cd}

\usepackage[T1]{fontenc}
\usepackage{eucal} 
\usepackage{enumerate}

\usepackage[longnamesfirst,numbers]{natbib}
\usepackage[bookmarks,colorlinks=true,citecolor=blue]{hyperref}
\theoremstyle{plain}
\newtheorem{thm}{Theorem}[section]

\newtheorem*{thm1.1}{Theorem 1.1}
\newtheorem*{thm1.2}{Theorem 1.2}

\newtheorem{lem}[thm]{Lemma}
\newtheorem*{lem*}{Lemma}
\newtheorem{cor}[thm]{Corollary}
\newtheorem*{cor*}{Corollary}
\newtheorem{prop}[thm]{Proposition}
\newtheorem*{prop*}{Proposition}
 
\theoremstyle{definition}
\newtheorem{defn}[thm]{Definition}
\newtheorem*{defn*}{Definition} 

\theoremstyle{remark}

\numberwithin{equation}{section}

\newcommand{\Nat}{\ensuremath{\mathbb{N}}}

\newcommand{\Rat}{\ensuremath{\mathbb{Q}}}
\newcommand{\Real}{\ensuremath{\mathbb{R}}}
\newcommand{\Cant}{\ensuremath{2^{\Nat}}}

\newcommand{\Str}[1][<\Nat]{\ensuremath{2^{#1}}}

\newcommand{\Sle}{\ensuremath{\subset}}

\newcommand{\Cyl}[1]{\ensuremath{\llbracket #1 \rrbracket}}

\newcommand{\Ceil}[1]{\ensuremath{\lceil #1 \rceil}}

\newcommand{\Rest}[1]{\ensuremath{ \lceil_{#1}}}
\newcommand{\Op}[1]{\ensuremath{\operatorname{#1}}}

\newcommand{\Conc}{\ensuremath{\mbox{}^\frown}}
\newcommand{\Estr}{\ensuremath{\varnothing}}

\newcommand{\Hmeas}{\ensuremath{\mathcal{H}}}
\newcommand{\Hdim}{\ensuremath{\dim_H}}

\renewcommand{\and}{\ensuremath{\:\&\:}}

\newcommand{\LoK}{\ensuremath{\underline{\kappa}}}

\newcommand{\Len}[1]{\ensuremath{\Op{len}(#1)}}

 \DeclareMathOperator{\K}{K}
\DeclareMathOperator{\KM}{KM}

\title{Information vs Dimension -- an Algorithmic Perspective}

\author{Jan Reimann}

\address{Department of Mathematics  \\
Penn State University, University Park 
}
\email{jan.reimann@psu.edu}


\thanks{Preprint version of an article published as Chapter 4 in \emph{Structure and randomness in computability and set theory, edited by D.\ Cenzer, C.\ Porter, J.\ Zapletal}, World Scientific, 2020.\\ \url{https://doi.org/10.1142/10661}}

\begin{document}

\begin{abstract}
This paper surveys work on the relation between fractal dimensions and
algorithmic information theory over the past thirty years. It covers the basic
development of prefix-free Kolmogorov complexity from an information theoretic
point of view, before introducing Hausdorff measures and dimension along with
some important examples. The main goal of the paper is to motivate and develop
the informal identity ``entropy = complexity = dimension'' from first
principles. The last section of the paper presents some new observations on
multifractal measures from an algorithmic viewpoint. 
\end{abstract}

\maketitle

\section{Introduction and Preliminaries}\label{sec:preliminaries}

Starting with the work by \citet{Ryabko:1984a,Ryabko:1984a} and
\citet{Staiger:1981a,Staiger:1989a,Staiger:1993a} in the 1980s, over the past 30
years researchers have investigated a strong relation between fractal dimension
and algorithmic information theory. At the center of this relation is a
\emph{pointwise} version of Hausdorff dimension (due to
\citet{Lutz:2000a,Lutz:2003a}), that is, a notion of dimension that is defined
for individual points in a space instead of subsets. This is made possible by
effectivizing the notion of measure (in the sense of \citet{Martin-Lof:1966a}),
which restricts the collection of nullsets to a countable family, thereby
allowing for singleton sets not to be null. Such singletons are considered
\emph{random} in this framework. The correspondence between randomness and the
various flavors of Kolmogorov complexity (a cornerstone of the theory of
algorithmic randomness) then re-emerges in form of an asymptotic information
density that bears close resemblance to the entropy of a stationary process.
Effective dimension makes the connections between entropy and fractal dimension
that have been unearthed by many authors (arguably starting with
\citet{Billingsley:1965a} and \citet{Furstenberg:1967a}) very transparent.
Moreover, effective dimension points to new ways to compute fractal dimensions,
by investigating the algorithmic complexities of single points in a set (e.g.\
\citet{Lutz:2017a}), in particular when studied in relation to other
\emph{pointwise} measures of complexity, such as irrationality exponents in
diophantine approximation \citep{Becher:2017a}.

This survey tries to present these developments in one coherent narrative.
Throughout, we mostly focus on Hausdorff dimension, and on a small fraction of
results on effective dimension and information, in order to streamline the
presentation. For other notions such as packing or box counting dimension, refer
to \citep{Falconer:2003a,Cutler:1993a} for the classical theory, and to
\citep{Reimann:2004a,Terwijn:2004a,Downey-Hirschfeldt:2010a,Staiger:2007a} for
many results concerning effective dimension.

The outline of the paper is as follows. In
Section~\ref{sec:information_measures}, we cover the basic theory of information
measures, introducing entropy and Kolmogorov complexity, highlighting the common
foundations of both. In Section~\ref{sec:hdim}, we briefly introduce Hausdorff
measures and dimension, along with some fundamental examples. Arguably the
central section of the paper,
Section~\ref{sec:hausdorff_dimension_and_information} develops the theory of
effective dimension, establishes the connection with asymptotic Kolmogorov
complexity, and describes the fundamental instance of entropy vs dimension --
the Shannon-Macmillan-Breiman Theorem. Finally, in
Section~\ref{sec:multifractal_measures}, we present some new observations
concerning multifractal measures. It turns out that the basic idea of pointwise
dimension emerges here, too.

It is the goal of this paper to be accessible for anyone with a basic background
in computability theory and measure theory. Where the proofs are easy to give
and of reasonable length, we include them. In all other cases we refer to the
literature.

\section{Information Measures and Entropy} 
\label{sec:information_measures}
Suppose \(X\) is a finite, non-empty set. We choose an element $a \in X$ and
want to communicate our choice to another person, through a binary channel,
i.e.\ we can only transmit bits $0$ and $1$. How many bits are needed to
transmit our choice, if nothing else is known about $X$ but its cardinality? To
make this possible at all, we assume that we have agreed on a numbering of $X =
\{a_1, \dots, a_n\}$ that is known to both parties. We can then transmit the
\emph{binary code} for the index of $a$, which requires at most $\log_2 n =
\log_2 |X|$ bits. Of course, if we had an $N$-ary channel available, we would
need $\log_N |A|$ bits. We will mostly work with binary codes, so $\log$ will
denote the binary logarithm $\log_2$.

Often our choice is guided by a probabilistic process, i.e.\ we have $X$ given
as a \emph{discrete random variable} $X$ has countably many outcomes, and we
denote the range of $X$ by $\{a_0, a_1, a_2, \ldots\}$. Suppose that we
repeatedly choose from $X$ at random and want to communicate our choice to the
other party. Arguably the central question of information theory is:
\begin{quote}
	\em How do we minimize the expected number of bits we need to transmit to
  communicate our choice?
\end{quote}

Depending on the distribution of the random variable, we can tweak our coding
system to optimize the length of the code words.

\begin{defn} \label{def:bin_code}
	A \emph{binary code} for a countable set $X$ is a one-one function $c: \{x_i
	\colon i \in \Nat\} \to \Str$.
\end{defn}

Here, $\Str$ denotes the set of all finite binary strings. In the
non-probabilistic setting, when $X$ is a finite set, we essentially use a
\emph{fixed-length code} $c:X \to \{0,1\}^n$, by mapping
\[
	x_i \mapsto \text{ binary representation of } i.
\]
In the probabilistic setting, if the distribution of the random variable $X$
does not distinguish between the outcomes, i.e.\ if $X$ is equidistiributed,
this will also be the best we can do.

However, if the possible choices have very different probabilities, we could
hope to save on the \emph{expected} code word length
\[
	\sum_{i} \Len{c(a_i)} \cdot \mathbb{P}(X = a_i),
\]
where $\Len{c(a_i)}$ is the length of the string $c(a_i)$, by assigning shorts
code words to $a$ of high probability. The question then becomes how small the
expected code word length can become, and how we would design a code to minimize
it.

\medskip Another way to approach the problem of information transmission is by
trying to measure the \emph{information gained} by communicating only partial
information. Suppose our random variable $X$ is represented by a partition \[
[0,1] = \stackrel{\cdot}{\bigcup_i} X_i, \] where each $X_i$ is an interval such
that $\mathbb{P}(X = a_i) = |X_i|$, i.e.\ the lengths of the intervals mirror
the distribution of $X$. Pick, randomly and uniformly, an $x \in [0,1]$. If we
know which $X_i$ $x$ falls in, we gain knowledge of about the first $-\log
|X_i|$ bits of the binary expansion of $x$. Therefore, on average, we gain \[
H(X) = - \sum \mathbb{P}(X=a_i) \cdot \log \mathbb{P}(X=a_i) = -\sum |X_i| \cdot
\log |X_i| \] bits of information. We put $0 \cdot \log 0 =: 0$ to deal with
$\mathbb{P}(X=a_i) = 0$. $H(X)$ is called the \emph{entropy} of $X$.  We will
apply $H$ not only to random variables, but also to measures in general. If
$\vec{p} = (p_1, \dots, p_n)$ is a finite probability vector, i.e.\ $p_i \geq 0$ and
\(
   \sum_i p_i = 1,
 \) 
then we write $H(\vec{p})$ to denote $\sum_i p_i \cdot \log(p_i)$. Similarly, for $p \in [0,1]$, $H(p)$ denotes $p\cdot \log p + (1-p)\cdot \log (1-p)$.

The emergence of the term $-\log \mathbb{P}$ is no coincidence. Another, more
axiomatic way to derive it is as follows. Suppose we want to measure the
information gain of an event by a function $I: [0,1] \to \Real^{\geq 0}$, that
is, $I$ depends only on the probability of an event. We require the following
properties of $I$:
\begin{enumerate}[({I}1)]
\item $I(1) = 0$. An almost sure event gives us no new information.
\item $I$ is decreasing. The lower the probability of an event, the more
  information we gain from the knowledge that it occurred.
\item $I(x \cdot y) = I(x) + I(y)$. If $X$ and $Y$ are independent, then
  $\mathbb{P}(X \cap Y) = \mathbb{P}(X) \cdot \mathbb{P}(Y)$, and hence
  $I(\mathbb{P}(X \cap Y)) = I(\mathbb{P}(X)) + I(\mathbb{P}(Y))$. In other
  words, information gain is additive for independent events.
\end{enumerate}

\begin{prop}
  If $I: [0,1] \to \Real^{\geq 0}$ is a function satisfying (I1)-(I3), then
  there exists a constant $c$ such that
  \begin{equation*}
        I(x) = -c \log(x).
  \end{equation*}
\end{prop}
In this sense, $-\log$ is the only reasonable information function, and the
entropy $H$ measures the expected gain in information.

It is also possible to axiomatically characterize entropy directly. Many such
characterizations have been found over the years, but the one by
\citet{Khinchin:1957a} is arguably still the most popular one.

We will see
next how the entropy of $X$ is tied to the design of an optimal code.


\subsection{Optimal Prefix Codes} 
\label{sub:optimal_prefix_codes}

We return to the question of how to find an optimal code, i.e.\ how to design a
binary code for a random variable $X$ such that its average code word length is
minimal. While a binary code guarantees that a single code word can be uniquely
decoded, we would like to have a similar property for \emph{sequences of code
words}.

\begin{defn}
	Given a binary code $c$ for a countable set $X$, its \emph{extension} $c^*$ to
    $X^{< \Nat}$ is defined as
	\[
		c^*(x_0 x_1\dots x_n ) = c(x_0) \Conc c(x_1) \Conc \dots \Conc c(x_n),
	\]
	where $\Conc$ denotes concatenation of strings. We say a binary code $c$ is
  \emph{uniquely decodable} if its extension is one-one.
\end{defn}

One way to ensure unique decodability is to make $c$ \emph{prefix free}. A set
$S \subseteq \Str$ is prefix free if no two elements in $S$ are prefixes of one
another. A code is prefix free if its range is. It is not hard to see that the
extension of a prefix free code is indeed uniquely decodable.

Binary codes live in $\Str$. The set $\Str$ is naturally connected to the space
of infinite binary sequences $\Cant$. We can put a metric on $\Cant$ by letting
\[
  d(x,y) =  \begin{cases}
      2^{-N} & \text{ if } x \neq y \text{ and } N = \min{i \colon x(i) \neq y(i)} \\ 
      0 & \text{ otherwise}.
  \end{cases}
\]
Given $\varepsilon > 0$, the $\varepsilon$-ball $B(x,\varepsilon)$ around $x$ is
the so-called \emph{cylinder set}
\[
  \Cyl{\sigma} = \{ y \in \Cant \colon \sigma \Sle y  \},
\]
where $\Sle$ denotes the prefix relation between strings (finite or infinite),
and $\sigma = x\Rest{n}$ is the length-$n$ initial segment of $x$ with $n =
\lceil -\log \varepsilon \rceil$. 

Hence any string $\sigma$ corresponds to a basic open cylinder $\Cyl{\sigma}$
with diameter $2^{-\Len{\sigma}}$. This induces a Borel measure $\lambda$ on
$\Cant$, $\lambda \Cyl{\sigma} = 2^{-\Len{\sigma}}$, which is the natural
analogue to Lebesgue measure on $[0,1]$.

\medskip
Which code lengths are possible for prefix free codes? The \emph{Kraft
  inequality} gives a fundamental restraint.

\begin{thm}[Kraft inequality] \label{thm:Kraft}
	Let $S \subseteq \Str$ be prefix free. Then 
	\[
		\sum_{\sigma \in S} 2^{-\Len{\sigma}} \leq 1.
	\]
	Conversely, given any sequence $l_0, l_1, l_2, \dots$ of non-negative integers
	satisfying
	\[
		\sum_i 2^{-l_i} \leq 1,
	\]
	there exists a prefix-free set $\{\sigma_0, \sigma_1, \dots \} \subseteq \Str$
  of code words such that $\Len{\sigma_i} = l_i$.
\end{thm}

\begin{proof}[Proof sketch] ($\Rightarrow$) Any prefix free set $S$ corresponds
  to a disjoint union of open cylinders
  \begin{equation*}
    U = \stackrel{\cdot}{\bigcup_{\sigma \in S}} \Cyl{\sigma}.
  \end{equation*}
  The Lebesgue measure $\lambda$ of $\Cyl{\sigma}$ is $2^{-\Len{\sigma}}$. Hence
  by the additivity of measures,
  \begin{equation*}
    1 \geq \lambda(U) = \sum_{\sigma \in S} 2^{-\Len{\sigma}}.
  \end{equation*}

  ($\Leftarrow$) We may assume the $l_i$ are given in non-decreasing order:
  \[
  	l_0 \leq l_1 \leq \dots,
  \]
  which implies the sequence $(2^{-l_i})$ is non-increasing. We construct a
  prefix code by ``shaving off'' cylinder sets from the left. We choose as
  $\sigma_0$ the leftmost string of length $l_0$, i.e.\ $0^{l_0}$. Suppose we
  have chosen $\sigma_0, \dots, \sigma_k$ such that $\Len{\sigma_i} = l_i$ for
  $i = 0,\dots, k$. The measure we have left is
  \begin{equation*}
    1 - (2^{-l_0}+ \cdots + 2^{-l_k}) \geq 2^{-l_{k+1}},
  \end{equation*}
  since $\sum_i 2^{-l_i} \leq 1$. And since $2^{-l_0} \geq \dots \geq 2^{-l_k}
  \geq 2^{-l_{k+1}}$, the remaining measure is a multiple of $2^{-l_{k+1}}$.
  Therefore, we can pick $\sigma_{k+1}$ to be the leftmost string of length
  $l_{k+1}$ that does not extend any of the $\sigma_0, \dots, \sigma_k$.
     
  This type of code is known as the \emph{Shannon-Fano code}. 

\end{proof}

Subject to this restraint, what is the minimum expected length of a code for a
random variable $X$? Suppose $X$ is discrete and let $p_i =\mathbb{P}(X = x_i)$.
We want to find code lengths $l_i$ such that
\[
  L = \sum p_i l_i
\]
is minimal, where the $l_i$ have to satisfy     
\[
	\sum 2^{-l_i} \leq 1, 
\]
the restriction resulting from the Kraft inequality.

The following result, due to Shannon, is fundamental to information theory.
\begin{thm}\label{thm:shannon_coding}
 Let $L$ be the expected length of a binary prefix code of a discrete random
 variable $X$. Then
\[
	L \geq H(X).
\]
Furthermore, there exists a prefix code for $X$ whose expected code length does
not exceed $H(X)+1$.
\end{thm}

To see why no code can do better than entropy, note that
\begin{align*}
	L - H(X) & = \sum p_i l_i + \sum p_i \log(p_i) \\
				& = - \sum p_i \log 2^{-l_i} + \sum p_i \log(p_i).
\end{align*}

Put $c = \sum 2^{-l_i}$ and $r_i = 2^{-l_i}/c$. $\vec{r}$ is the probability
distribution induced by the code lengths. Then
\begin{align*}
	L - H(X) & = \sum p_i \log p_i - \sum p_i \log \left (\frac{2^{-l_i}}{c} c\right ) \\
				& = \sum p_i \log \frac{p_i}{r_i} - \log c \\
\end{align*}
Note that $c \leq 1$ and hence $\log c \leq 0$. The expression
\[
   \sum p_i \log \frac{p_i}{r_i}
\]
is known as the \emph{relative entropy} or \emph{Kulback-Leibler (KL) distance}
$D(\vec{p} \parallel \vec{r})$ of distributions $\vec{p}$ and $\vec{r}$. It is
not actually a distance as it is neither symmetric nor does it satisfy the
triangle inequality. But it nevertheless measures the difference between
distributions in some important aspects. In particular, $D(\vec{p} \parallel
\vec{r}) \geq 0$ and is equal to zero if and only if $\vec{p} = \vec{r}$. This
can be inferred from Jensen's inequality.

\begin{thm}[Jensen's inequality] If $f$ is a convex function and $X$ is a random
  variable, then
  \[
    \mathbb{E} f(X) \geq f(\mathbb{E}X).
  \]
\end{thm}

In particular, it follows that $L-H(X) \geq 0$. This can also be deduced
directly using the concavity of the logarithm in the above sums, but both
KL-distance and Jensen's inequality play a fundamental, almost ``axiomatic''
role in information theory, so we reduced the problem to these two facts.

\medskip
To see that there exists a code with average code length close to entropy, let
\[
  l_i := \Ceil{\log p_i}.
\]
Then $\sum 2^{-l_i} \leq \sum p_i = 1$, and by Theorem~\ref{thm:Kraft}, there
exists a prefix code $\vec{\sigma}$ such that $\Len{\sigma_i} = l_i$. Hence $x_i
\mapsto \sigma_i$ will be a prefix code for $X$ with expected code word length
\[
  \sum p_i l_i = \sum p_i \Ceil{\log p_i} \leq \sum p_i (\log p_i + 1) = H(X) + 1.
\]

\medskip

\subsection{Effective coding}
\label{sub:eff-coding}

The transmission of information is an inherently computational act. Given $X$,
we should be able to assign code words to outcomes of $X$ effectively. In other
words, if $\{x_i\}$ collects the possible outcomes of $X$, and $\{\sigma_i\}$ is
a prefix code for $X$, then the mapping $i \mapsto \sigma_i$ should be
computable.

The construction of a code is implicit in the proof of the Kraft inequality
(Theorem~\ref{thm:Kraft}), and is already of a mostly effective nature. The only
non-effective part is the assumption that the word lengths $l_i$ are given in
increasing order. However, with careful book keeping, one can assign code words
in the order they are requested. This has been observed independently by~
\citet{Levin:2010a}, \citet{Schnorr:1973a}, \citet{Chaitin:1975b}.  

\begin{thm}[Effective Kraft inequality] \label{thm:eff-kraft}
  Let $L: \Nat \to \Nat$ be a computable mapping such that
  \[
    \sum_{i \in \Nat} 2^{-L(i)} \leq 1.
  \]
  Then there exists a computable, one-one mapping $c: \Nat \to \Str$ such that
  the image of $c$ is prefix free, and for all $i$, $\Len{c(i)} = L(i)$.
\end{thm}

A proof can be found, for example, in the book
by~\citet{Downey-Hirschfeldt:2010a}. Let us call $c$ an effective prefix-code.
As $c$ is one-one, given $\sigma \in \operatorname{range}(c)$, we can compute
$i$ such that $c(i) = \sigma$. And since the image of $c$ is prefix-free, if we
are given a sequence of concatenated code words 
\[ 
  \sigma_1\Conc \sigma_2 \Conc \sigma_3 \dots =  c(i_1) \Conc c(i_2) \Conc c(i_3) \Conc \dots,
\]
we can effectively recover the source sequence $i_1, i_2, i_3\dots$. Hence,
every effective prefix code $c$ comes with an effective \emph{decoder}.

\medskip
Let us therefore switch perspective now: Instead of mapping events to code words,
we will look at functions mapping code words to events.  An \emph{effective
prefix decoder} is a partial computable function $d$ whose domain is a
prefix-free subset of $\Str$. Since every partial recursive function is
computable by a Turing machine, we will just use the term \emph{prefix-free
machine} to denote effective prefix decoders. The range of a prefix-free machine
will be either $\Nat$ or $\Str$. Since we can go effectively from one set to the
other (by interpreting a binary string as the dyadic representation of a natural
number or by enumerating $\Str$ effectively), in this section it does not really
matter which one we are using. In later sections, however, when we look at
infinite binary sequences, it will be convenient to have $\Str$ as the range.  

If $M$ is a prefix-free machine and $M(\sigma) = m$, we can think of $\sigma$ as
an $M$-code of $m$. Note that by this definition, $m$ can have multiple codes,
which is not provided for in our original definition of a binary code
(Definition~\ref{def:bin_code}). So a certain asymmetry arises when switching
perspective. We can eliminate this asymmetry by considering only the
\emph{shortest} code. The length of this code will then, with respect to the
given machine, tell us how far the information in $\sigma$ can be
\emph{compressed}.  

\begin{defn}
  Let $M$ be a prefix-free machine. The $M$-complexity $\K_M(m)$ of a string
  $\sigma$ is defined as
  \[
    \K_M(m) = \min \{ \Len{\sigma} \colon M(\sigma) = m\},
  \]
  where $\K_M(m) = \infty$ if there is no $\sigma$ with $M(\sigma) = m$.
\end{defn}

$\K_M$ may not be computable anymore, as we are generally not able to decide whether
$M$ will halt on an input $\sigma$ and output $m$. But we can \emph{approximate
$\K_M$ from above}: Using a bootstrapping procedure, we run $M$ simultaneously
on all inputs. Whenever we see that $M(\sigma)\downarrow = m$, we compare
$\sigma$ with our current approximation of $\K_M(m)$. If $\sigma$ is shorter,
$\Len{\sigma}$ is our new approximation of $\K_M(m)$. Eventually, we will have
found the shortest $M$-code for $m$, but in general, we do not know whether our
current  approximation is actually the shortest one.

Alternatively, we can combine the multiple codes into a single one
probabilistically. Every prefix-free machine $M$ induces a distribution on
$\Nat$ by letting
\begin{equation}
  Q_M(m) = \sum_{M(\sigma)\downarrow = m} 2^{-\Len{\sigma}}. 
\end{equation}  
$Q_M$ is not necessarily a probability distribution, however, for it is not
required that the domain of $M$ covers all of $\Str$. Nevertheless, it helps to
think of $Q_M(m)$ as the probability that $M$ on a randomly chosen input halts
and outputs $m$.

Any function $Q: \Nat \to \Real^{\geq 0}$ with
\[
  \sum_m Q(m) \leq 1
\]
is called a \emph{discrete semimeasure}. The term 'discrete' is added to
distinguish these semimeasures from the \emph{continuous semimeasures} which
will be introduced in Section~\ref{sec:multifractal_measures}. In this section,
all semimeasures are discrete, so we will refer to them just as 'semimeasures'. 

If a semimeasure $Q$ is induced by a prefix-free machine, it moreover has the
property that for any $m$, the set
\[
  \{ q \in \Rat \colon  q < Q(m) \}
\]
is recursively enumerable, uniformly in $m$. In general, semimeasures with this
property are called \emph{left-enumerable}.

We can extend the effective Kraft Inequality to left-enumerable semimeasures. We
approximate a code enumerating better and better dyadic lower bounds for $Q$. As
we do not know the true value of $Q$, we cannot quite match the code lengths
prescribed by $Q$, but only up to an additive constant. 

\begin{thm}[Coding Theorem] \label{thm:coding}
  Suppose $Q$ is a left-enumerable semimeasure. Then there exists a prefix-free
  machine $M_Q$ such that for some constant $d$ and for all $m \in \Nat$,
  \[
    \K_{M_Q}(m) \leq -\log Q(m) + d
  \]
\end{thm}
For a proof see \citet{Li-Vitanyi:2008a}. Since the shortest $M$-code
contributes $2^{-\K_M(m)}$ to $Q_M(m)$, it always holds that
\[
  \K_{M_Q}(m) \leq \K_M(m) + d
\]
for all $m$ and for some constant $d$. However, if we consider universal
machines, the two processes will essentially yield the same code.

\begin{defn}
  A prefix-free machine $U$ is \emph{universal} if for each partial computable
  prefix-free function $\varphi: \Str \to \Nat$, there is a string
  $\sigma_\varphi$ such that for all $\tau$,
  \[
    U(\sigma_\varphi \Conc \tau) = \varphi(\tau)
  \]
\end{defn}
Universal prefix-free machines exist. The proof is similar to the proof that
universal (plain) Turing machines exist (see
e.g.~\citep{Downey-Hirschfeldt:2010a}).

Fix a universal prefix-free machine $U$ and define 
\[
  \K(m) = \K_U(m).
\]
$\K(m)$ is called the \emph{prefix-free Kolmogorov complexity} of $m$. The
universality of $U$ easily implies the following.

\begin{thm}[Invariance Theorem]\label{thm:K-inv}
   For any prefix-free machine $M$ there exists a constant $d_M$ such that for
   all $m$,
   \[
     \K(m) \leq \K_M(m) + d_M.
   \]
\end{thm} 

Theorems~\ref{thm:coding} and~\ref{thm:K-inv} together imply that up to an
additive constant,
\[
  \K = \K_U \leq \K_{Q_U} \leq \K_U = \K, 
\]
that is, up to an additive constant, \emph{both coding methods agree} (see
Figure~\ref{fig:compl-vs-semimeasure}).

\medskip
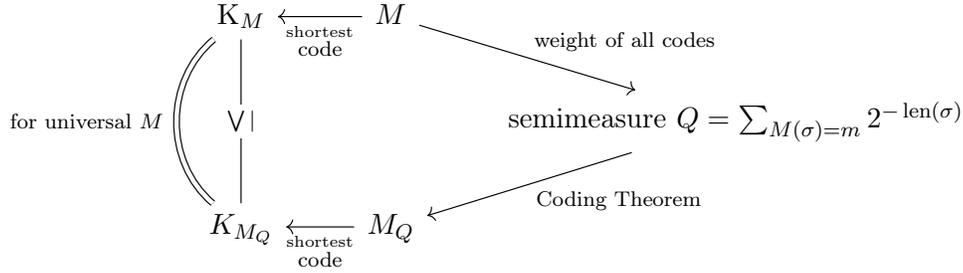
\begin{figure}
\begin{center}
\begin{tikzcd}
 \K_M \arrow[dd, dash, "\bigvee | " description] \arrow[dd, equal, bend
 right=50, "\text{for universal } M \;"'] & M  \arrow[l,
 "\overset{\text{shortest}}{\text{code}}"] \arrow[rd, "\text{weight of all
 codes}"] &   \\
 \bigvee \!\! | &  & \text{semimeasure } Q =  \sum_{M(\sigma) = m}
 2^{-\Len{\sigma}} \arrow[ld, "\text{Coding Theorem}"]\\
 K_{M_Q} & M_Q  \arrow[l, "\overset{\text{shortest}}{\text{code}}"]& 
\end{tikzcd}
\caption{Relating complexity and semimeasures} \label{fig:compl-vs-semimeasure}
\end{center}
\end{figure}

\medskip

The Coding Theorem also yields a relation between $\K$ and Shannon entropy $H$.
Suppose $X$ is a random variable with a computable distribution $p$. As $\K$ is
the length of a prefix code on $\Str$, Shannon's lower bound
(Theorem~\ref{thm:shannon_coding}) implies that $H(X) \leq \sum_m \K(m) p(m)$.
On the other hand, it is not hard to infer from the Coding Theorem that, up to a
multiplicative constant,  $2^{-\K(m)}$ dominates every computable probability
measure on $\Nat$. Therefore,
\[
    H(X) = \sum_m p(m) (-\log p(m)) \geq \sum_m p(m) \K(m) -d    
\] 
for some constant $d$ depending only on $p$. In fact, the constant $d$ depends
only on the length of a program needed to compute $p$ (which we denote $\K(p)$).

\begin{thm}[\citet{Cover:1989a}] There exists a constant $d$ such that for every
  discrete random variable $X$ with computable probability distribution $p$ on
  $\Nat$ with finite Shannon entropy $H(X)$, 
  \[
    H(X) \leq \sum_\sigma \K(\sigma)p(\sigma) \leq H(X) + \K(p) + d. 
  \]
\end{thm}

In other words, probabilistic entropy of a random variable X taking outcomes in
N is equal, within an additive constant, to the expected value of complexity of
these outcomes. This becomes interesting if $\K(p)$ is rather low compared to
the length of strings $p$ puts its mass on, for example, if $p$ is the uniform
distribution on strings of length $n$.

\bigskip

\section{Hausdorff dimension}
\label{sec:hdim}

The origins of Hausdorff dimension are geometric. It is based on finding an
optimal scaling relation between the diameter and the measure of a set.

Let $(Y,d)$ be a separable metric space, and suppose $s$ is a non-negative real
number. Given a set $U \subseteq Y$, the diameter of $U$ is given as
\[
  d(U) = \sup \{ d(x,y) \colon x,y \in U \}.
\]
Given $\delta >0$, a family $(U_i)_{i \in I}$ of subsets of $Y$ with $d(U_i)
\leq \delta$ for all $i \in I$ is called a \emph{$\delta$-family}. 

The \emph{$s$-dimensional outer Hausdorff measure} of a set $A \subseteq Y$ is
defined as 
\[
  \Hmeas^s(A) = \lim_{\delta \to 0} \left ( \inf \left \{ \sum d(U_i)^s \colon A \subseteq \bigcup U_i, \; (U_i) \text{ $\delta$-family}\right \} \right ).
\]
Since there are fewer $\delta$-families available as $\delta \to 0$, the limit
above always exists, but may be infinite. If $s = 0$, $\Hmeas$ is a counting
measure, returning the number of elements of $A$ if $A$ is finite, or $\infty$
if $A$ is infinite. If $Y$ is a Euclidean space $\Real^n$, then $\Hmeas^n$
coincides with Lebesgue measure, up to a multiplicative constant.

It is not hard to see that one can restrict the definition of $\Hmeas^s$ to
$\delta$-families of open sets and still obtain the same value for $\Hmeas^s$.

In $\Cant$, it suffices to consider only $\delta$-families of cylinder sets,
since for covering purposes any set $U \subseteq \Cant$ can be replaced by
$\Cyl{\sigma}$, where $\sigma$ is the longest string such that $\sigma \Sle x$
for all $x \in U$ (provided $U$ has more than one point). Then $d(\Cyl{\sigma})
= d(U)$. It follows that $\delta$-covers in $\Cant$ can be identified with sets
of strings $\{\sigma_i \colon i \in \Nat\}$, where $\Len{\sigma_i} \geq  \lceil
\log \delta \rceil$ for all $i$.

In $\Real$, one can obtain a similar simplification by considering the
\emph{binary net measure} $\mathcal{M}^s$, which is defined similar to
$\Hmeas^s$, but only coverings by dyadic intervals of the form $[r2^{-k},
(r+1)2^{-k})$ are permitted. For the resulting outer measure $\mathcal{M}^s$, it
holds that, for any set $E \subseteq \Real$,
\begin{equation} \label{equ:H-vs-net}
  \Hmeas^s E \leq \mathcal{M}^s E \leq 2^{s+1} \Hmeas^s E.
\end{equation}
Instead of dyadic measures, one may use other $g$-adic net measures and obtain a
similar relation as in~\eqref{equ:H-vs-net}, albeit with a different
multiplicative constant (see e.g.~\citep{Falconer:2003a}).

For each $s$, $\Hmeas$ is a \emph{metric outer measure}, that is, if $A,B
\subseteq Y$ have positive distance ($\inf \{d(x,y) \colon x \in A, y \in B\}
>0$), then
\[
    \Hmeas^s(A \cup B) = \Hmeas^s(A) + \Hmeas^s(B).
\]
This implies that all Borel sets are $\Hmeas^s$-measurable~\citep{Rogers:1970a}.

\medskip
It is not hard to show that 
\begin{equation}\label{equ:hmeas-up-0}
  \Hmeas^s(A) < \infty \; \Rightarrow \; \Hmeas^t(A) = 0 \text{ for all $t > s$},
\end{equation}
and 
\begin{equation} \label{equ:hmeas-donw-infty}
  \Hmeas^s(A) > 0 \; \Rightarrow \; \Hmeas^t(A) = \infty \text{ for all $t < s$}.  
\end{equation}
Hence the number
\[
  \Hdim(A) = \inf \{ s \colon \Hmeas(A) = 0 \} = \sup\{ s \colon \Hmeas(A) = \infty \}
\]
represents the ``optimal'' scaling exponent (with respect to diameter) for
measuring $A$. $\dim_H(A)$ is called the \emph{Hausdorff dimension} of $A$.
By~\eqref{equ:H-vs-net}, working in $\Real$, we can replace $\Hmeas^s$ by the
binary net measure $\mathcal{M}^s$ in the definition of Hausdorff dimension. On of the theoretical advantages of Hausdorff dimension over other concepts that measure fractal behavior (such as box counting dimension) is its \emph{countable stability}.

\begin{prop}\label{pro:Hdim-stable}
  Given a countable family $\{A_i\}$ of sets, it holds that
  \[
    \Hdim \bigcup A_i = \sup \left \{ \Hdim A_i \colon i \in \Nat \right \}. 
  \]
\end{prop}

\medskip

\subsubsection*{Examples of Hausdorff dimension} 

\mbox{}

\smallskip
(1) Any finite set has
Hausdorff dimension zero.

(2) Any set of non-zero $n$-dimensional Lebesgue measure in $\Real^n$ has
    Hausdorff dimension $n$. Hence all ``regular'' geometric objects of full
    topological dimension (e.g.\ squares in $\Real^2$, cylinders in $\Real^3$)
    have full Hausdorff dimension, too.

(3) The \emph{middle-third Cantor set}
  \[
    C = \{ x \in [0,1]\colon x = \sum_{i \geq 1} t_i\cdot 3^{-i}, \; t_i\in {0,2} \}
  \]
  has Hausdorff dimension $\ln(2)/\ln(3)$.

(4) Inside Cantor space, we can define a ``Cantor-like'' set by letting
  \[
    D = \{ x \in \Cant \colon x(2n) = 0 \text{ for all $n$}\}.
  \]
  Then $\Hdim (D) = \frac{1}{2}$.

\medskip
\subsubsection*{Generalized Cantor sets} 
\label{ssub:generalized_cantor_sets}
Later on, a more general form of Cantor set will come into play. Let $m \geq 2$
be an integer, and suppose $0 < r < 1/m$. Put $I^{0}_0 = [0,1]$. Given intervals
$I_n^{0}, \ldots, I_n^{m^n-1}$, define the intervals $I_{n+1}^0,\ldots,
I_{n+1}^{m^{n+1}-1}$ by replacing each $I_{n}^j$ by $m$ equally spaced intervals
of length $r |I_{n}^j|$. The Cantor set $C_{m,r}$ is then defined as
\[
   C_{m,r} = \bigcap_n \bigcup_i I_n^i.
\] 
It can be shown that $$\Hdim C_{m,r} = \frac{\ln(m)}{-\ln(r)},$$ see
\citep[Example 4.5]{Falconer:2003a}. In fact, it holds that $0 <
\Hmeas^{\ln(m)/-\ln(r)} (C_{m,r})< \infty$. An upper bound on
$\Hmeas^{\ln(m)/-\ln(r)}$ can be obtained by using the intervals $I_n^{0},
\ldots, I_n^{m^n-1}$ of the same level as a covering. To show
$\Hmeas^{\ln(m)/-\ln(r)} (C_{m,r}) >0$, one can show a mass can be distributed
along $C_{m,r}$ that disperses sufficiently fast. A \emph{mass distribution} on
a subset $E$ of a metric space $Y$ is a measure on $Y$ with
support\footnote{Recall that the support of a measure on a metric space $Y$ is
the largest closed subset $F$ of $Y$ such that for every $x \in F$, if $U \ni x$
is an open neighborhood of $x$, $\mu U > 0$} contained in $Y$ such that $0 < \mu
E < \infty$.

\begin{lem}[Mass Distribution Principle] \label{lem:mdp}
  Suppose $Y$ is a metric space and $\mu$ is a mass distribution on a set  $E
  \subseteq Y$. Suppose further that for some $s$ there are numbers $c > 0$ and
  $\varepsilon > 0$ such that 
  \[
    \mu U \leq c \cdot d(U)^s
  \]
  for all sets $U$ with $d(U)< \varepsilon$. Then $\Hmeas^s\geq \mu(E)/c$.
\end{lem}

For a proof, refer to \citep{Falconer:2003a}. The mass distribution principle
has a converse, \emph{Frostman's Lemma}, which we will address in
Section~\ref{sec:hausdorff_dimension_and_information}.

\section{Hausdorff dimension and information} 
\label{sec:hausdorff_dimension_and_information}

A 1949 paper by \citet{Eggleston:1949a} established a connection between
Hausdorff dimension and entropy. Let $\vec{p} = (p_1, \dots, p_N)$ be a
probability vector, i.e.\ $p_i \geq 0$ and $\sum p_i = 1$. Given $(p_1, \dots,
p_N)$, $i < N$, and $x \in [0,1]$, let
\begin{align*}
  \#(x,i,n) = & \# \text{ occurrences of digit $i$ among the first $n$ digits} \\
    & \quad \text{ of the $N$-ary expansion of $x$ }
\end{align*}
and
\[
  D(\vec{p}) = \left \{ x\in [0,1] \colon \lim_{n\to \infty} \frac{\#(x,i,n)}{n}
    = p_i \text { for all $i \leq N$}\right \}.
\]

\begin{thm}[Eggleston] \label{thm:eggleston}
For any $\vec{p}$,
\[
  \dim_H D(\vec{p}) = - \sum_{i} p_i \log_N p_i = \log_2 N \cdot H(\vec{p}).
\]
\end{thm}
Using an effective version of Hausdorff dimension, we can see that Eggleston’s
result is closely related to a fundamental theorem in information theory.

\subsection{Hausdorff dimension and Kolmogorov complexity} 
\label{sub:effective_dimension}


In the 1980s and 1990s, a series of papers by
Ryabko~\citep{Ryabko:1984a,Ryabko:1986a} and
Staiger~\citep{Staiger:1981a,Staiger:1989a,Staiger:1993a,Staiger:1998a}
exhibited a strong connection between Hausdorff dimension, Kolmogorov
complexity, and entropy. A main result was that if a set has high Hausdorff
dimension, it has to contain elements of high incompressibility ratio. In this
section, we focus on the space $\Cant$ of infinite binary sequences. It should
be clear how the results could be adapted to sequence spaces over arbitrary
finite alphabets. We will later comment how these ideas can be adapted to
$\Real$.

\begin{defn}
  Given $x \in \Cant$, we define the \emph{lower incompressibility ratio} of $x$
  as
  \[
    \LoK(x) = \liminf_{n \to \infty} \frac{K(x\Rest{n})}{n} 
   \]
  respectively. The lower incompressibility ratio of a subset $A \subseteq
  \Cant$ is defined as
  \[
    \LoK(A) = \sup \{ \LoK(x) \colon x \in A \} 
  \]
\end{defn}

\begin{thm}[Ryabko] \label{thm:ryabko-bound}
  For any set $A \subseteq \Cant$,
  \[
    \Hdim A \leq \LoK(A).
  \]
\end{thm}

Staiger was able to identify settings in which this upper bound is also a lower
bound for the Hausdorff dimension, in other words, the dimension of a set equals
the least upper bound on the incompressibility ratios of its individual members.
In particular, he showed it holds for any $\Sigma^0_2$ (lightface) subset of
$\Cant$.

\begin{thm}[Staiger] \label{thm:staiger-correspondence}
  If $A \subseteq \Cant$ is $\Sigma^0_2$, then
  \[
    \Hdim A = \LoK(A).
  \]
\end{thm}

Using the countable stability of Hausdorff dimension, this result can be extended to arbitrary countable unions of $\Sigma^0_2$ sets.

The ideas underlying Ryabko's and Staiger's results were fully exposed when
Lutz~\citep{Lutz:2000a,Lutz:2003a} extended Martin-Löf's concept of randomness
tests for individual sequences~\cite{Martin-Lof:1966a} to Hausdorff measures.
Lutz used a variant of martingales, so called $s$-gales\footnote{A variant of this concept appeared, without explicit reference to Hausdorff dimension, already in Schnorr's fundamental book on algorithmic randomness \citep{Schnorr:1971a}}. In this paper, we will
follow the approach using variants of Martin-Löf tests. We will define a rather
general notion of Martin-Löf test which can be used both for probability
measures and Hausdorff measures.

\subsection{Effective nullsets} 
\label{sub:effective_nullsets}

\begin{defn}
  A \emph{premeasure} on $\Cant$ is a function $\rho: \Str \to [0,\infty)$. 
\end{defn}

\begin{defn}
  Let $\rho$ be a computable premeasure on $\Cant$. A \emph{$\rho$-test} is a
  sequence $(W_n)_{n \in \Nat}$ of sets $W_n \subseteq \Str$ such that for all
  $n$,
   \[
      \sum_{\sigma \in W_n} \rho(\sigma) \leq 2^{-n}.
   \]
   A set $A \subseteq \Cant$ is \emph{covered by a $\rho$-test} $(W_n)$ if 
   \[
     A \subseteq \bigcap_n \bigcup_{\sigma \in W_n} \Cyl{\sigma},
   \]
   in other words, $A$ is contained in the $G_\delta$ set induced by $(W_n)$.
\end{defn} 

Of course, whether this is a meaningful notion depends on the nature of $\rho$.
We  consider two families of premeasures, \emph{Hausdorff premeasures} of the
form $\rho(\sigma) = 2^{-\Len{\sigma}s}$, and \emph{probability premeasures},
which satisfy 
\begin{gather} 
   \rho(\Estr)  = 1, \label{equ:prob_premeasure1} \\
   \rho(\sigma)  = \rho(\sigma\Conc 0) + \rho(\sigma\Conc 1). \label{equ:prob_premeasure2}
 \end{gather} 
All premeasures considered henceforth are assumed to be either Hausdorff
premeasures or probability premeasures. The latter we will often simply refer to
as \emph{measures}.

Hausdorff premeasures are completely determined by the value of $s$, and hence
in this case we will speak of $s$-tests. By the Carathéodory Extension Theorem,
probability premeasures extend to a unique Borel probability measure on $\Cant$.
We will therefore also identify Borel probability measures $\mu$ on $\Cant$ with
the underlying premeasure $\rho_\mu(\sigma) = \mu \Cyl{\sigma}$.

Even though the definition of Hausdorff measures is more involved than that of
probability measures, if we only are interested in Hausdorff nullsets (which is
the case if we only worry about Hausdorff dimension), it is sufficient to
consider $s$-tests.

\begin{prop}
   $\Hmeas^s(A) = 0$ if and only if $A$ is covered by an $s$-test.
\end{prop}

The proof is straightforward, observing that any set $\{W_n\}$ for which
$\sum_{\sigma \in W_n} 2^{-\Len{\sigma}s}$ is small admits only cylinders with
small diameter.

Tests can be made \emph{effective} by requiring the sequence $\{W_n\}$ be
uniformly recursively enumerable in a suitable
representation~\citep{Day-Miller:2013a,Reimann-Slaman:2015a}. To avoid having to
deal with issues arising from representations of non-computable measures, we
restrict ourselves here to computable premeasures.

We call the resulting test a \emph{Martin-Löf $\rho$-test}.

If $A$ is covered by a Martin-Löf $\rho$-test, we say that $A$ is
  \emph{effectively $\rho$-null}. There are at most countably many Martin-Löf
  $\rho$-tests. Since for the premeasures we consider countable unions of null
  sets are null, for every premeasure $\rho$ there exists a non-empty set of
  points not covered by any $\rho$-test. In analogy to the standard definition
  Martin-Löf random sequences for Lebesgue measure $\lambda$, it makes sense to
  call such sequences \emph{$\rho$-random}.

\begin{defn}
  Given a premeasure $\rho$, a sequence $x \in \Cant$ is \emph{$\rho$-random} if
  $\{x\}$ is not effectively $\rho$-null.
\end{defn}

If $\rho$ is a probability premeasure, we obtain the usual notion of Martin-Löf
randomness for probability measures. For Hausdorff measures, similar to
\eqref{equ:hmeas-up-0} and \eqref{equ:hmeas-donw-infty}, we have the following
(since every $s$-test is a $t$-test for all rational $t \geq s$).

\begin{prop}
  Suppose $s$ is a rational number. If $x$ is not $s$-random, then it is not
  $t$-random for all rational $t > s$. And if $x$ is $s$-random, then it is
  $t$-random for all rational $0 <t < s$.
\end{prop}

We can therefore define the following \emph{pointwise} version of Hausdorff
dimension.

\begin{defn}[Lutz] The \emph{effective} or \emph{constructive} Hausdorff
  dimension of a sequence $x \in \Cant$ is defined as
  \[
    \Hdim x = \inf \{ s > 0 \colon x \text{ not $s$-random }\}.
  \]
\end{defn}

We use the notation $\Hdim x$ to stress the pointwise analogy to (classical)
Hausdorff dimension.

\subsection{Effective dimension and Kolmogorov complexity} 
\label{sub:effective_dimension_and_kolmogorov_complexity}

Schnorr~\citep[see][]{Downey-Hirschfeldt:2010a} made the fundamental observation
that Martin-Löf randomness for $\lambda(\sigma) = 2^{-\Len{\sigma}}$ and
incompressibility with respect to prefix-free complexity coincide.
\citet{Gacs:1980a} extended this to other computable measures.

\begin{thm}[\citet{Gacs:1980a}] \label{thm:mu-rand-K}
  Let $\mu$ be a computable measure on $\Cant$. Then $x$ is $\mu$-random if and
  only of there exists a constant $c$ such that
 \[
   \forall n \; K(x\Rest{n}) \geq -\log \mu\Cyl{x\Rest{n}}.
  \] 
\end{thm}

One can prove a similar result for $s$-randomness.

\begin{thm} \label{thm:s-rand_K}  
  A sequence $x \in \Cant$ is $s$-random if and only if there exists a $c$ such
  that for all $n$,
  \[
    K(x\Rest{n}) \geq sn - c.
  \]   
\end{thm}    

In other words, being $s$-random corresponds to being incompressible only to a
factor of $s$. For this reason $s$-randomness has also been studied as a notion
of \emph{partial randomness} (for example, \citep{Calude-Staiger-Terwijn:2006a})

\begin{proof}
  ($\Leftarrow$) Assume $x$ is not $s$-random. Thus there exists an $s$-test
$(W_n)$ covering $x$. Define functions $m_n: \Str \to \Real$ by 
\begin{align*}
        m_n(\sigma) & = \begin{cases}
                n 2^{-\Len{\sigma}s-1} & \text{if $\sigma \in W_n$}, \\
                0                          & \text{otherwise,}
        \end{cases} \\
        \intertext{and let}
        m(\sigma) & = \sum_{n=1}^\infty m_n(\sigma).
\end{align*}
All $m_n$ and thus $m$ are enumerable from below. Furthermore,
\[
  \sum_{\sigma \in \Str} m(\sigma) = \sum_{\sigma \in \Str} \sum_{n=1}^\infty m_n(\sigma)
      = \sum_{n=1}^\infty n \sum_{\sigma \in W_n} 2^{-\Len{w}s-1}   
      \leq \sum_{n=1}^\infty \frac{n}{2^{n+1}} \leq 1. 
\]
Hence $m$ is a semimeasure and enumerable from below. By the Coding Theorem
(Theorem~\ref{thm:coding}), there exists a constant $d$ such that for all
$\sigma$,
\[
   K(\sigma) \leq -\log m(\sigma) + d.
\] 
Now let $c > 0$ be any constant. Set $k = \Ceil{2^{c+d+1}}$. As $x$ is covered
by $(W_n)$ there is some $\sigma \in C_k$ and some $l \in \Nat$ with $\sigma =
x\Rest{l}$. This implies
\[
  K(x\Rest{l}) = K(\sigma) \leq  -\log m(\sigma) + d = -\log(k 2^{\Len{\sigma}s} ) +d \leq s l - c. 
\]
As $c$ was arbitrary, we see that $K(x\Rest{n})$ does not have a lower bound of
the form $s n - \text{constant}$.

\medskip
($\Rightarrow$) Assume for every $k$ there exists an $l$ such that $K(x\Rest{l})
< sl - k$. Let  
\begin{equation*}
  W_k := \{ \sigma \colon K(\sigma) \leq s \Len{\sigma} - k\}.
\end{equation*}
Then $(W_k)$ is uniformly enumerable, and each $W_k$ covers $x$. Furthermore,
\begin{equation*}
  \sum_{\sigma \in W_k} 2^{-\Len{\sigma}s} \leq \sum_{\sigma \in W_k} 2^{-K(\sigma)-k}  \leq 2^{-k}.
\end{equation*}
Therefore, $(W_k)$ is an $s$-test witnessing that $x$ is not $s$-random.
\end{proof}

As a corollary, we obtain the ``fundamental theorem'' of effective dimension
\citep{Mayordomo:2002a,Staiger:2005a}.  

\begin{cor}\label{cor:fundamental}
  For any $x \in \Cant$,
  \[
    \Hdim x = \LoK(x).
  \]
\end{cor}

\begin{proof}
  Suppose first $\Hdim x < s$. This means $x$ is not $s$-random and by
  Theorem~\ref{thm:s-rand_K}, for infinitely many $n$, $K(x\Rest{n}) \leq sn$.
  Thus $\LoK(x) \leq s$. On the other hand, if $\LoK(x) < s$, there exists
  $\varepsilon > 0$ such that $K(x\Rest{n}) < (s-\varepsilon)n$ for infinitely
  many $n$. In particular, $sn - K(x\Rest{n})$ is not bounded from below by a
  constant, which means $x$ is not $s$-random. Therefore, $\Hdim x \leq s$.
\end{proof}

As a further corollary, we can describe the correspondence between Hausdorff
dimension and Kolmogorov complexity (Theorem~\ref{thm:staiger-correspondence})
as follows. 

\begin{cor}
  For every $\Sigma^0_2$ $A \subseteq \Cant$, 
  \[
    \dim_H A = \sup_{x \in A} \dim_H x.
  \]
\end{cor}

The general idea of these
results is that if a set of a rather simple nature in terms of definability,
then set-wise dimension can be described as the supremum of pointwise
dimensions. For more complicated sets, one can use relativization to obtain a
suitable correspondence. One can relativize the definition of effective
dimension with respect to an oracle $z \in \Cant$. Let us denote the
corresponding notion by $\dim_H^z$. If we relativize the notion of Kolmogorov
complexity as well, we can prove a relativized version of
Corollary~\ref{cor:fundamental}. 

\citet{Lutz:2017a} were able to prove a most general set-point correspondence
principle that is based completely on relativization, in which the Borel
complexity does not feature anymore. 

\begin{thm}[Lutz and Lutz] \label{thm:set-point-dim}
  For any set $A \subseteq \Cant$, 
  \[
    \dim_H A = \min_{z \in \Cant} \sup_{x \in A} \dim_H^z x.
  \]
\end{thm}
In other words, the Hausdorff dimension of a set is the minimum among its
pointwise dimensions, taken over all relativized worlds.

These point-vs-set principles are quite remarkable in that they have no direct
classical counterpart. While Hausdorff dimension is \emph{countably stable} (Proposition~\ref{pro:Hdim-stable}),
stability does not extend to arbitrary unions, since singleton sets always have
dimension $0$.

\subsection{Effective dimension vs randomness} 
\label{sub:effective_dimension_vs_randomness}

Let $\mu$ be a computable probability measure on $\Cant$. By
Theorem~\ref{thm:mu-rand-K}, for a $\mu$-random $x \in \Cant$, the prefix-free
complexity of $x\Rest{n}$ is bounded from below by $\log \mu \Cyl{x\Rest{n}}$.
If we can bound $\log \mu \Cyl{x\Rest{n}}$  asymptotically from below by a bound
that is linear in $n$, we also obtain a lower bound on the effective dimension
of $x$. This can be seen as an effective/pointwise version of the mass
distribution principle: If a point ``supports'' measure that decays sufficiently
fast along it, we can bound its effective dimension from below.

Consider, for example, the \emph{Bernoulli measure} $\mu_p$ on $\Cant$. Given $p
\in [0,1]$, $\mu_p$ is induced by the premeasure
\[
  \rho_p(\sigma\Conc 1) = \rho_p(\sigma)\cdot p \qquad \rho_p(\sigma\Conc 0) = \rho_p(\sigma)\cdot (1-p).
\]
As we will see later, if $x$ is $\mu$-random,
\begin{equation} \label{equ:smb-Bernoulli}
  \lim_{n \to \infty} \frac{-\log \mu_p \Cyl{x\Rest{n}}}{n} = H(p),
\end{equation}
and therefore, $\Hdim x \geq H(p)$ for any $\mu$-random $x$. It follows that the
``decay'' of $\mu$ along $x$ is bounded by $2^{-H(p)n}$.

Does the reverse direction hold, too? If $x$ has positive effective dimension,
is $x$ random for a suitably fast dispersing measure? 

\begin{thm}[\citet{Reimann:2008a}] Suppose $x \in \Cant$ is such that $\Hdim x >
  s$. Then there exists a probability measure $\mu$ on $\Cant$ and a constant
  $c$ such that for all $\sigma \in \Str$,
  \[
    \mu \Cyl{\sigma} \leq c 2^{-sn}.
  \]
\end{thm} 

This is an effective/pointwise version of \emph{Frostman's
Lemma}~\citep{Frostman:1935a}, which in Cantor space essentially says that if a
Borel set has Hausdorff dimension $> s$, it supports a measures that disperses
at least as fast as $2^{-sn}$ (up to a multiplicative constant).

\subsection{The Shannon-Macmillan-Breiman Theorem} 
\label{sub:the_shannon_macmillan_breiman_theorem}

In \eqref{equ:smb-Bernoulli}, we already stated a special case of the
\emph{Shannon-McMillan-Breiman Theorem}, also known as the \emph{Asymptotic
Equipartition Property} or \emph{Entropy Rate Theorem}. It is one of the central
results of information theory. 

Let $T: \Cant \to \Cant$ be the \emph{shift map}, defined as
\[
  T(x)_i = x_{i+1},
\]
where $y_i$ is the $i$-th bit of $y \in \Cant$. A measure $\mu$ on $\Cant$ is
\emph{shift-invariant} if for any Borel set $A$, $\mu T^{-1}(A) = \mu A$. A
shift-invariant measure $\mu$ is \emph{ergodic} if $T^{-1}(A) = A$ implies $\mu
A = 0$ or $\mu A =1$. Bernoulli measures $\mu_p$ are an example of
shift-invariant, ergodic measures. For background on ergodic measures on
sequence spaces, see~\citep{Shields:1996a}.

\begin{thm} \label{thm:SMB} Let $\mu$ be a shift-invariant ergodic measure on
  $\Cant$. Then there exists a non-negative number $h$ such that almost surely,
  \[
    \lim_{n \to \infty} - \frac{\log \mu\Cyl{x\Rest{n}} }{n} = h
  \]
\end{thm}
For a proof, see for example \citep{Shields:1996a}. The number $h$ is called the
\emph{entropy rate} of $\mu$ and also written as $h(\mu)$. It is also possible
to define the entropy of the underlying $\{0,1\}$-valued process. First, let
\begin{equation*}
  H(\mu^{(n)}) = -\sum_{\sigma \in A^n} \mu\Cyl{\sigma} \log \mu\Cyl{\sigma}.
\end{equation*}
One can show that this is \emph{subadditive} in the sense that
\begin{equation*}
  H(\mu^{(n+m)}) \leq H(\mu^{(n)}) + H(\mu^{(m)}),
\end{equation*}
which implies that
\begin{equation*}
  H(\mu) = \lim \frac{1}{n} H(\mu^{(n)})
\end{equation*}
exists. $H(\mu)$ is called the \emph{process entropy} of $\mu$. It is clear that
for i.i.d.\ $\{0,1\}$-valued processes, $H(\mu)$ agrees with the entropy of the
distribution on $\{0,1\}$ as defined in Section~\ref{sec:information_measures}.

Entropy rate is a local measure, as it follows the behavior of a measure along a
typical point, while process entropy captures the entropy over finer and finer
partitions of the whole space $\Cant$. Process entropy in turn is a special case
of a general definition of entropy in measure-theoretic dynamical systems, known
as \emph{Kolmogorov-Sinai entropy} (see for example \citep{Walters:1982a}).

The Shannon-Macmillan-Breiman Theorem states that the pointwise entropy rate
\begin{equation*}
  \lim_{n \to \infty} - \frac{\log \mu\Cyl{x\Rest{n}} }{n}
\end{equation*}
not only exists almost surely, but also is constant up to a set of measure zero.
Furthermore, it can be shown that this constant entropy rate coincides with the
process entropy $H(\mu)$. This fundamental principle has been established in
much more general settings, too (e.g. for amenable groups, see
\citep{Ornstein-Weiss:1983a},\citep{Lindenstrauss:2001a}).

\subsection{The effective SMB-theorem and dimension}
\label{sec:eff-SMB}

We have already seen a connection between randomness and complexity in
Theorems~\ref{thm:mu-rand-K} and~\ref{thm:s-rand_K}.

If $\mu$ is a computable ergodic measure, this immediately connects the
complexity of infinite sequences to entropy via the Shannon-Macmillan-Breiman
Theorem.  Since $\mu$-almost every sequence is $\mu$-random,  we obtain that
almost surely,
\[
  \Hdim x = \liminf_{n \to \infty} \frac{K(x\Rest{n})}{n} \geq \frac{-\log \mu\Cyl{x\Rest{n}}}{n} = H(\mu).
\]
\citet{Zvonkin-Levin:1970a} and independently \citet{Brudno:1982a} showed that
for $\mu$-almost every sequence,
\[
  \lim_{n \to \infty} \frac{K(x\Rest{n})}{n} \text{ exists},
\]
which implies in turn the asymptotic compression ratio of almost every real
equals the metric entropy of $\mu$. Analyzing the proof of the SMB-Theorem by
\citet{Ornstein-Weiss:1983a}, \citet{Vyugin:1998a} was able to establish,
moreover, that for \emph{all} $\mu$-random sequences $x$,
\[
     \limsup_{n \to \infty} \frac{K(x\Rest{n})}{n} = H(\mu).
\] 
Finally, \citet{Hoyrup:2011a} established that the $\liminf$ also equals
$H(\mu)$, thereby showing that the SMB-Theorem is effective in the sense of
Martin-Löf randomness.

\begin{thm}[Effective Shannon-Macmillan-Breiman Theorem \citep{Vyugin:1998a,Hoyrup:2011a}] \label{thm:eff_SMB}
  Let $\mu$ be a computable, ergodic, shift-invariant measure. Then, for each
  $x\in \Cant$ that is random with respect to $\mu$, 
  \[
    \dim_H x = \lim_{n \to \infty}- \frac{\log \mu\Cyl{x\Rest{n}} }{n} =  H(\mu).
  \]
\end{thm}

Note that effective dimension as a pointwise version of Hausdorff dimension is a
natural counterpart to the pointwise nature of entropy rate, the asymptotic
compression ratio (in terms of Kolmogorov complexity) forming the bridge between
the two, or, to state it somewhat simplified:
\begin{center}
  \emph{dimension $=$ complexity $=$ entropy}
\end{center}

\medskip

\subsection{Subshifts} 
\label{sub:subsection_name}

The principle above turns out to be quite persistent, even when one passes from
effective dimension of sequences (points) to (classical) dimension of sets. A
\emph{subshift} of $\Cant$ is a closed subset $X \subseteq \Cant$ invariant
under the shift map on $\Cant$.

Subshifts are a \emph{topological dynamical system}, but they can also carry an
invariant measure. There is a topological variant of entropy, which for
subshifts is given as
\[
  h_{\Op{top}}(X) = \lim_{n \to \infty} \frac{\log |\{ \sigma \colon |\sigma| = n \and \sigma \Sle X \}|}{n},
\]
that is, $h_{\Op{top}}$ measures the relative numbers of strings present in $X$
as the length goes to infinity.

\citet{Furstenberg:1967a} showed that for subshifts $X \subseteq \Cant$,
\[
  h_{\Op{top}}(X) = \dim_H X.
\]
\citet{Simpson:2015a} extended this result to multidimensional one- and
two-sided subshifts. Furthermore, he established that for \emph{any} such
subshifts, the coincidence between entropy, complexity, and dimension is
complete in the sense that
\[
  h_{\Op{top}}(X) = \dim_H X = \LoK(X).
\]
Most recently, \citet{Day:2017a} further extended this identity to computable
subshifts of $A^G$, where $G$ is a computable amenable group with computable
bi-invariant tempered Følner sequence. 

\citet{Staiger:1989a,Staiger:1993a} showed that Furstenberg's result also holds for other families of  sets: closed sets definable by finite automata, and for $\omega$-powers of languages definable by finite automata.

\subsection{Application: Eggleston's theorem} 
\label{sub:application_eggleston_s_theorem}

Using the dimension-complexity-entropy correspondence, we can also give a short
proof of Eggleston's Theorem (Theorem~\ref{thm:eggleston}). We focus on the
binary version. The proof for larger alphabets is similar. 

Let $p \in [0,1]$ be computable, and consider the Cantor space equivalent of
$D_p$,
\[
  \overline{D}_p = \{ x \in \Cant \colon \lim \frac{\#(x,0,n)}{n} = p\}, 
\]
where $\#(x,i,n)$ now denotes the number of occurrences of digit $i$ among the
first $n$ digits of $x$.

All $\mu_p$-random sequences satisfy the Law
of Large Numbers. Therefore, there exists a $\Pi^{0}_1$ class $P \subset
\overline{D}_p$ consisting of only $\mu_p$-random sequences. By
Theorem~\ref{thm:staiger-correspondence} and Theorem~\ref{thm:eff_SMB}
\[
  \dim_H D_p \geq \dim_H P = \LoK(P) = H(\mu_p) = H(p).
\]
On the other hand, for every $x \in D_p$,
\[
  \LoK(x) \leq H(p),
\]
as we can always construct a two-part code for $x \in D_p$, by giving the number
of $1$'s in $x\Rest{n}$, and then the position of $x\Rest{n}$ in a lexicographic
enumeration of all binary strings of length $n$ with that number of $1$'s. As
$\log_2 \binom{n}{k} \approx n H(k/n)$, this gives the desired upper bound for
$\LoK(x)$. Therefore, by Theorem~\ref{thm:ryabko-bound},
\[
  \dim_H D_p \leq H(p).
\]
For non-computable $p$, one can use relativized versions of the results used above.

\medskip
\subsubsection*{Cantor space vs the real line}
The reader will note that Eggleston's Theorem, as originally stated, applies to
the unit interval $[0,1]$, while the version proved above is for Cantor space
$\Cant$. It is in fact possible to develop the theory of effective dimension for
real numbers, too. One way to do this is to use net measures $\mathcal{M}^s$ as
introduced in Section ~\ref{sec:hdim}. Dyadic intervals correspond directly to
cylinder sets in $\Cant$ via the finite-to-one mapping
\[
  x\in \Cant \mapsto \sum_n x_n 2^{-(n+1)}.
\]
If we identify a real with its dyadic expansion (which is unique except for
dyadic rationals), we can speak of Kolmogorov complexity of (initial segments
of) reals etc. Furthermore, the connection between effective dimension and
Kolmogorov complexity carries over to the setting in $[0,1]$. For a thorough
development of effective dimension in Euclidean (and other) spaces,
see~\citep{Lutz-Mayordomo:2008a}.
\medskip

\section{Multifractal measures} 
\label{sec:multifractal_measures}

A measure can have a fractal nature by the way it spreads its mass. To prove a
lower bound on the Hausdorff dimension of a Cantor set, we distribute a mass
uniformly along it (and appeal to the mass distribution principle,
Lemma~\ref{lem:mdp}). In other words, we define a measure supported on a
fractal. But a measure can exhibit fractal properties also through \emph{the way
it spreads its mass over its support}, which is not necessarily uniform. Think
of earthquake distributions along fault systems. The fault systems themselves
are usually of a fractal nature. Moreover, there often seems to be a
non-uniformity in the distribution of earthquakes along fault lines. Some ``hot
spots'' produce much more earthquakes than other, more quiet sections of a fault
(see for example \citet{Kagan:1980a,Kagan:1991a,Mandelbrot:1989a}). The
underlying mathematical concept has become known as \emph{multifractality of
measures}. It is widely used in the sciences today. It would go far beyond the
scope of this survey to introduce the various facets of multifractal measures
here. Instead, we focus on one aspect of multifractality that appears to be a
natural extension of the material presented in
Section~\ref{sec:hausdorff_dimension_and_information}. The interested reader may
refer to a forthcoming paper \citep{Reimann:b}.

\subsection{The dimension of a measure} 
\label{sub:the_dimension_of_a_measure}

Unless explicitly noted, in the following \emph{measure} will always mean Borel
probability measure on $[0,1]$.

First, we define the \emph{dimension of a measure}, which reflects the fractal
nature of its support.

\begin{defn}[\citet{Young:1982a}] Given a Borel probability measure $\mu$ on
  $\Real$, let
  \[
    \Hdim \mu = \inf \{ \Hdim E \colon \mu E = 1 \}
  \]
\end{defn}

For example, if we distribute a unit mass uniformly along a Cantor set
$C_{m,r}$, we obtain a measure $\mu_{m,r}$ with
\[
  \Hdim \mu_{m,r} = -\ln(m)/\ln(r).
\]
Instead by a single number, one can try to capture the way a measure spreads its
mass among fractal sets by means of a distribution.

\begin{defn}[\citet{Cutler:1993a}] The \emph{dimension distribution} of a
  measure $\mu$ is given by
  \[
    \mu_{\dim} [0,t] = \sup  \{ \mu D \colon \Hdim D \leq t, \; D \subseteq \Real \text{ Borel} \}.
  \]
\end{defn}

As $\Hdim$ is countably stable (see
Section~\ref{sub:effective_dimension_and_kolmogorov_complexity}), $\mu_{\dim}$
extends to a unique Borel probability measure on $[0,1]$. The dimension of a
measure can be expressed in terms of $\mu_{\dim}$:
\[
  \Hdim \mu = \inf \{ \alpha \colon \mu_{\dim}[0,\alpha] = 1 \}. 
\]
For many measures, the dimension distribution does not carry any extra
information. For example, if $\lambda$ is Lebesgue measure on $[0,1]$, the
dimension distribution is the \emph{Dirac point measure} $\delta_1$, where
$\delta_\alpha$ is defined as  
\[
   \delta_\alpha A = \begin{cases}
    1 &  \alpha \in A, \\
    0 &   \alpha \not\in A.
  \end{cases}
\]
Measures whose dimension distribution is a point measure are called \emph{exact
dimensional}. This property also to uniform distributions on Cantor sets.

\begin{prop}
  Let $\mu_{m,r}$ be the uniform distribution along the Cantor set $C_{m,r}$.
  Then $\mu_{\dim} = \delta_\alpha$, where $\alpha = -\ln(m)/\ln(r)$.  
\end{prop}

It is not completely obvious that $\mu_{m,r}$ has no mass on any set of
Hausdorff dimension $< -\ln(m)/\ln(r)$. One way to see this is by connecting the
dimension distribution of a measure to its \emph{pointwise dimensions}, which we
will introduce next.

\subsection{Pointwise dimensions} 
\label{sub:pointwise_dimensions}

\begin{defn}
  Let $\mu$ be a probability measure on a separable metric space. The
  \emph{(lower) pointwise dimension} of $\mu$ at a point $x$ is defined as 
  \[
    \delta_\mu(x) = \liminf_{\varepsilon \to 0} \frac{\log \mu B(x,\varepsilon)}{\log \varepsilon}.
  \]
  (Here $B(x,\varepsilon)$ is the $\varepsilon$-ball around $x$.)
\end{defn}

Of course, one can also define the upper pointwise dimension by considering
$\limsup$ in place of $\liminf$, which is connected to packing measures and
packing dimension similar to the way lower pointwise dimension is connected to
Hausdorff measures and dimension. As this survey is focused on Hausdorff
dimension, we will also focus on the lower pointwise dimensions and simply speak
of ``pointwise dimension'' when we mean lower pointwise dimension.

We have already encountered the pointwise dimension of a measure when looking at
the Shannon-Macmillan-Breiman Theorem  (Theorem~\ref{thm:SMB}), which says that
for ergodic measures on $\Cant$ $\mu$-almost surely the pointwise dimension is
equal to the entropy of $\mu$.

We have also seen (Corollary~\ref{cor:fundamental}) that the effective dimension
of a real $x$ is its pointwise dimension with respect to the semimeasure
$\widetilde{Q}(\sigma) = 2^{-K(\sigma)}$.

We can make this analogy even more striking by considering a different kind of
universal semimeasure that gives rise to the same notion of effective dimension.

Recall that a semimeasure is a function $Q:\Nat \to \Real^{\geq 0}$ such that
$\sum Q(m) \leq 1$. Even when seen as a function on $\Str$, such a semimeasure
is of a discrete nature. In particular, it does not take into account the
partial ordering of strings with respect to the prefix relation. The notion of a
\emph{continuous} semimeasure does exactly that. As the prefix relation
corresponds to a partial ordering of basic open sets, and premeasures are
defined precisely on those sets, continuous semimeasures respects the structure
of $\Cant$. Compare the following definition with the properties of a
probability premeasure \eqref{equ:prob_premeasure1},
\eqref{equ:prob_premeasure2}. 

\begin{defn}
  \begin{enumerate}[(i)]
    \item    A \emph{continuous semimeasure} is a function $M: \Str \to
    \Real^{\geq 0}$ such that
  \begin{gather*} 
      M(\Estr) \leq 1, \\
      M(\sigma) \geq M(\sigma\Conc 0) + M(\sigma\Conc 1).
  \end{gather*}   

    \item   A continuous semimeasure $M$ is \emph{enumerable} if there exists a
    computable function $f:\Nat\times \Str \to \Real^{\geq 0}$ such that for all
    $n, \sigma$,
  \[
    f(n, \sigma) \leq f(n+1, \sigma) \quad \text{ and } \quad \lim_{n \to \infty} f(n,\sigma) = M(\sigma).
  \]

    \item  An enumerable continuous semimeasure $M$ is \emph{universal} if for
    every enumerable continuous semimeasure $P$ there exists a constant $c$ such
    that for all $\sigma$,
    \[
       P(\sigma) \leq c\, M(\sigma).
    \] 
  \end{enumerate}
\end{defn} 

Levin~\citep{Zvonkin-Levin:1970a} showed that a universal semimeasure exists.
Let us fix such a semimeasure $\mathbf{M}$. Similar to $K$, we can introduce
a complexity notion based on $\mathbf{M}$ by letting
\[
  \KM(\sigma) = -\log \mathbf{M}(\sigma).
\]
Some authors denote $\KM$ by $\operatorname{KA}$ and refer to it as \emph{a
priori complexity}. $\KM$ is closely related to $\K$. 

\begin{thm}[\citet{Gacs:1983a},\citet{Uspensky-Shen:1996a}] There exist
  constants $c,d$ such that for all $\sigma$,
  \[
    \KM(\sigma) - c \leq \K(\sigma) \leq \KM(\sigma) + \K(\Len{\sigma}) + d.
  \]
\end{thm}

When writing $K(n)$, we identify n with is binary representation (which is a
string of length approximately $\log(n)$). Since $\K(n)/n \to 0$ for $n \to
\infty$, we obtain the following alternative characterization of effective
dimension.

\begin{cor}
  For any $x \in \Cant$,
  \[
    \Hdim x = \liminf_{n \to \infty} \frac{\KM(x\Rest{n})}{n} = \liminf_{n \to \infty} \frac{-\log \mathbf{M}(x\Rest{n})}{n}.
  \]
\end{cor}

In other words, the effective dimension of a point $x$ in Cantor space (and
hence also in $[0,1]$) is its pointwise dimension with respect to a universal
enumerable continuous semimeasure $\mathbf{M}$.

$\KM$ ``smoothes'' out a lot of the complexity oscillations  $\K$ has. This has
a nice effect concerning the characterization of random sequences. On the one
hand, $\mathbf{M}$ is universal among enumerable semimeasures, hence in
particular among computable probability measures. Therefore, for any $x \in
\Cant$, and any computable probability measure $\mu$,
\[
  -\log \mathbf{M}(x\Rest{n}) - c  \leq -\log \mu\Cyl{x\Rest{n}} \quad \text{ for all $n$},
\]
where $c$ is a constant depending only on $\mu$.

On the other hand, \citet{Levin:1973b} showed that $x \in \Cant$ is random with
respect to computable probability measure $\mu$ if and only if for some constant
$d$,
\[
  \KM(x\Rest{n}) \geq -\log \mu \Cyl{x\Rest{n}} - d \quad \text{ for all $n$}.
\]
Thus, along $\mu$-random sequences, $-\log \mathbf{M}$ and $-\log \mu$ differ
by at most a constant. This gives us the following.

\begin{prop} \label{pro:pointwise-random}
  Let $\mu$ be a computable probability measure on $\Cant$. If $x \in \Cant$ is
  $\mu$-random, then
  \[
      \Hdim x = \delta_\mu(x).
  \]  
\end{prop}

Given a measure $\mu$ and $\alpha \geq 0$, let
\[
  D^\alpha_\mu = \{ x \colon \delta_\mu(x) \leq \alpha \}.
\]
If $\mu$ is computable, then since $\mathbf{M}$ is a universal semimeasure,
\[
  D^\alpha_\mu \subseteq \{x \colon \Hdim x \leq \alpha \}.
\]
Let us denote the set on the right hand side by $\dim_{\leq \alpha}$.

\begin{thm}[\citet{Cai-Hartmanis:1994a,Ryabko:1984a}] \label{thm:cai-hart}
  For any $\alpha \geq 0$,
  \[
    \Hdim (\dim_{\leq \alpha}) = \alpha.
  \]
\end{thm}

It follows that $\Hdim D^\alpha_\mu \leq \alpha$. This was first shown, for
general $\mu$, by \citet{Cutler:1990b,Cutler:1992a}. Cutler also characterized
the dimension distribution of a measure through its pointwise dimensions.

\begin{thm}[\citet{Cutler:1990b,Cutler:1992a}] For each $t \geq 0$,
  \[
      \mu_{\dim}[0,t] = \mu(D^t_\mu).
  \]  
\end{thm}
By the above observation, for computable measures, we can replace $D^\alpha_\mu$
by $\dim_{\leq t}$. The theorem explains why $\Hdim \mu_{m,r} =
\delta_{-\ln(m)/\ln(r)}$: Since the mass is distributed uniformly over the set
$C_{m,r}$, $\delta_{\mu_{m,r}}(x)$ is almost surely constant.

\subsection{Multifractal spectra} 
\label{sub:multifractal_spectra}

We have seen in the previous section that uniform distributions on Cantor sets
result in dimension distributions of the form of Dirac point measures
$\delta_\alpha$. We can define a more ``fractal'' measure by biasing the
distribution process. Given a probability vector $\vec{p} = (p_1, \ldots, p_m)$,
the measure $\mu_{m,r,\vec{p}}$ is obtained by splitting the mass in the
interval construction of $C_{m,r}$ not uniformly, but according to $\vec{p}$.
The resulting measure is still exact dimensional (see~\cite{Cutler:1993a} --
this is similar to the SMB Theorem). However, the measures $\mu_{m,r,\vec{p}}$
exhibit a rich fractal structure when one considers the \emph{dimension
spectrum} of the pointwise dimensions (instead of its dimension distribution).

\begin{defn}
  The \emph{(fine) multifractal spectrum} of a measure $\mu$ is defined as
  \[
    f_\mu(\alpha) = \Hdim \{ x \colon \delta_\mu(x) = \alpha\}.
  \]
\end{defn}

For the measures $\mu_{m,r,\vec{p}}$, it is possible to compute the multifractal
spectrum using the \emph{Legendre transform} (see~\citep{Falconer:2003a}). It is
a function continuous in $\alpha$ with  $f_\mu(\alpha) \leq \alpha$ and maximum
$\Hdim \mu$. In a certain sense, the universal continuous semimeasure
$\mathbf{M}$ is a ``perfect multifractal'', as, by the Cai-Hartmanis result,
$f_{\mathbf{M}}(\alpha) = \alpha$ for all $\alpha \in [0,1]$.

Moreover, the multifractal spectrum of a computable $\mu = \mu_{m,r,\vec{p}}$
can be expressed as a ``deficiency'' of multifractality against $\mathbf{M}$.

\begin{thm}[\citet{Reimann:b}] For a computable measure $\mu =
  \mu_{m,r,\vec{p}}$, it holds that
  \[
    f_\mu(\alpha) = \Hdim \left \{ x \colon \frac{\Hdim x}{\dim_\mu x} = \alpha \right \}.
  \]
\end{thm}
Here $\dim_\mu$ denotes the \emph{effective Billingsley dimension} of $x$ with
respect to $\mu$ (see~\citep{Reimann:2004a}).

\bibliography{rand-inf}
\bibliographystyle{abbrvnat}

\end{document}